\documentclass[11pt]{article}
\usepackage[pdftex]{graphicx}
\usepackage{amssymb,amsmath,latexsym, amsthm,amscd}
\usepackage{tikz} 
\usetikzlibrary{matrix,arrows}
\usepackage[margin=1.5in]{geometry} 
\usepackage[overload]{textcase}
\usepackage{hyperref}
  \hypersetup{
    colorlinks,
    citecolor=black,
    filecolor=black,
    linkcolor=black,
    urlcolor=black
}

\newtheorem{thm}{Theorem}[section]
\newtheorem{prop}[thm]{Proposition}

\newtheorem{cor}[thm]{Corollary}
\newtheorem{fact}[thm]{Fact}
\newtheorem{lem}[thm]{Lemma}

\begin{document}

\title{SUBGROUPS OF MOD(S) GENERATED BY $X \in \{(T_aT_b)^k,(T_bT_a)^k\}$ AND $Y \in \{T_a,T_b\}$}
\author{JAMIL MORTADA}
\date{\today}
\maketitle

\begin{abstract}
Suppose $a$ and $b$ are distinct isotopy classes of essential simple closed curves in an orientable surface $S$. Let $T_a$ and $T_b$ represent the respective Dehn twists along $a$ and $b$. In this paper, we study the subgroups of Mod(S) generated by $X$ and $Y$, where $X \in \{(T_aT_b)^k,(T_bT_a)^k\}$, $k \in \mathbb{Z}$, and $Y \in \{T_a,T_b\}$. For a large class of examples, we show that the subgroups $\langle X,Y \rangle$ and $\langle T_a,T_b \rangle$ are isomorphic. Moreover, we prove that $\langle X,Y \rangle = \langle T_a,T_b \rangle$ whenever $i(a,b) = 1$ and $k$ is not a multiple of three or $i(a,b) \geq 2$ and $k = \pm 1$. Further, we compute the index $[\langle T_a,T_b \rangle : \langle X,Y \rangle]$ when $\langle X,Y \rangle$ is a proper subgroup of $\langle T_a,T_b \rangle$.
\end{abstract}

\section{Introduction}
 \label{Introduction}
 
Let $S = S_{g,b}$ be a surface of genus $g$ and $b$ boundary components. Throughout this paper, we assume that $S$ is connected, orientable, compact, and of finite type. Denote by Mod(S) the mapping class group of $S$, which is the group of isotopy classes of orientation preserving homeomorphisms of $S$ which fix the boundary $\partial S$ pointwise. 
 \medskip

Let $a$ and $b$ represent distinct isotopy classes of essential simple closed curves in $S$, and let $T_a$ and $T_b$ be the respective Dehn twists along $a$ and $b$. In this paper, we investigate the subgroups $\langle X,Y \rangle$ of Mod(S), where $X \in \{(T_aT_b)^k,(T_bT_a)^k\}$, $k \in \mathbb{Z}$, and $Y \in \{T_a,T_b\}$. We compute $\langle X,Y \rangle$ based on $k$ and the geometric intersection numbers $i(a,b)$. In particular, we show that $\langle X,Y \rangle$ is isomorphic to one of the following groups: $\mathbb{Z}$, $\mathbb{Z}^2$, $\mathcal{B}_3$, $SL_2(\mathbb{Z})$, or $\mathbb{F}_2$. It turns out that $\langle X,Y \rangle \cong \langle T_a,T_b \rangle$ whenever $i(a,b) \neq 1$ and $k \neq 0$. Moreover, the two groups coincide whenever $i(a,b) = 1$ and $k$ is not a multiple of three or $i(a,b) \geq 2$ and $k = \pm 1$. In the first case, the group generated by $T_a$ and $T_b$ is isomorphic to $SL_2(\mathbb{Z})$ if $S$ is the torus $S_{1,0}$ and is isomorphic to the braid group $\mathcal{B}_3$ on three strands when $S \neq S_{1,0}$. In the second case, the group generated by $T_a$ and $T_b$ is isomorphic to the free group on two generators.
 \medskip
   
Consider two distinct isotopy classes $a$ and $b$ of essential simple closed curves in $S$. Denote by $T_a$ and $T_b$ the respective (left) Dehn twists along $a$ and $b$. Let $X \in \{(T_aT_b)^k,(T_bT_a)^k\}$, $k \in \mathbb{Z}$, and $Y \in \{T_a,T_b\}$. Denote by $G$ the subgroup of Mod(S) generated by $X$ and $Y$. The structure of $G$ is independent of the isotopy classes $a$ and $b$. Rather, $G$ depends only on $k$ and the geometric intersection $i(a,b)$. Since $i(a,b)$ is symmetric in $a$ and $b$, it follows that the subgroups $\langle (T_aT_b)^k,T_a \rangle$ and $\langle (T_bT_a)^k,T_b \rangle$ are isomorphic. Similarly, $\langle (T_aT_b)^k,T_b \rangle$ and $\langle (T_bT_a)^k,T_a \rangle$ are isomorphic. Thus, the structure of $G$ is symmetric with respect to $T_a$ and $T_b$, and it is enough to study $G$ modulo this symmetry. In other words, it suffices to consider $X \in \{(T_aT_b)^k,(T_bT_a)^k\}$ and fix $Y = T_a$ in order to investigate the group $G$. We prove the following theorem:

\begin{thm}[Main Theorem]\label{main theorem}
Suppose that $a$ and $b$ are distinct isotopy classes of essential simple closed curves in $S$. Let $T_a$ and $T_b$ denote the (left) Dehn twists along $a$ and $b$ respectively. Let $X \in \{(T_aT_b)^k,(T_bT_a)^k\}$, $k \in \mathbb{Z}$, and $Y = T_a$. Denote by $G$ the subgroups of Mod(S) generated by $X$ and $Y$.
\smallskip
\begin{list}{\labelitemi}{\leftmargin=1em}
\item If $k = 0$, then $G = \langle T_a \rangle \cong \mathbb{Z}$
\item If $k \neq 0$ and $i(a,b) = 0$, then $G = \langle T_a,T_b^k \rangle \cong \mathbb{Z}^2$. Moreover, $G$ has index $k$ in $\langle T_a,T_b \rangle$.
\item If $k \neq 0$ and $i(a,b) \geq 2$, then $G \cong \mathbb{F}_2$. Moreover, $G = \langle T_a,T_b \rangle$ when $k = \pm 1$ and $G$ is a subgroup of infinite index in $\langle T_a,T_b \rangle$ otherwise.
\item If $k \neq 0$ and $i(a,b) = 1$, then 

When $S = S_{1,0}$, 
\[
 G =
  \left\{
   \begin{array}{lll}
    \langle T_a,T_b \rangle \cong SL_2(\mathbb{Z}) & \mbox{if $k \not \equiv 0 \hspace{0.1cm} mod(3)$} \\
    \mathbb{Z}_2 \times \mathbb{Z} & \mbox{if $k \equiv 3 \hspace{0.1cm} mod(6)$}\\
    \langle T_a \rangle \cong \mathbb{Z} & \mbox{if $k \equiv 0 \hspace{0.1cm} mod(6)$}
   \end{array}
  \right.
\]
In the last two cases, $G$ has infinite index in $\langle T_a,T_b \rangle$.
\medskip

When $S \neq S_{1,0}$, 
\[
 G =
  \left\{
   \begin{array}{ll}
    \langle T_a,T_b \rangle \cong \mathcal{B}_3 & \mbox{if $k \not \equiv 0 \hspace{0.1cm} mod(3)$} \\
    \mathbb{Z}^2 & \mbox{if $k \equiv 0 \hspace{0.1cm} mod(3)$}
   \end{array}
  \right.
\]
In the second case, $G$ has infinite index in $\langle T_a,T_b \rangle$.
\end{list}
\end{thm}

\noindent \textbf{Acknowledgements.} I am very grateful to Sergio Fenley for carefully reading this paper and suggesting improvements. 

\section{Background on Dehn Twists}
 \label{Background on Dehn Twists}
 
This section provides a basic background about Dehn twists and some of their relevant properties. For more information, the reader is referred to [\ref{FM}], [\ref{Iv}], and [\ref{B}]. 
 \medskip
 
Let $\alpha$ be a simple closed curve in $S$ and let $N = N_{\epsilon}(\alpha)$ denote a regular neighborhood of $\alpha$. A left Dehn twist (with respect to the orientation of $S$) along $\alpha$ is a homeomorphism $T_{\alpha} : S \rightarrow S$ which is supported on $N$ and is the identity on the complement of $N$. If $\beta$ is an arc transverse to $\alpha$, then $T_{\alpha}$ affects $\beta$ by causing it to turn left near the intersection point, go once around $\alpha$, then proceed along $\beta$ as before. See Figure~\ref{Dehn twist} for an illustration. 
 \medskip
 
\begin{figure}[h]
	\centering
		\includegraphics[width=0.80\textwidth]{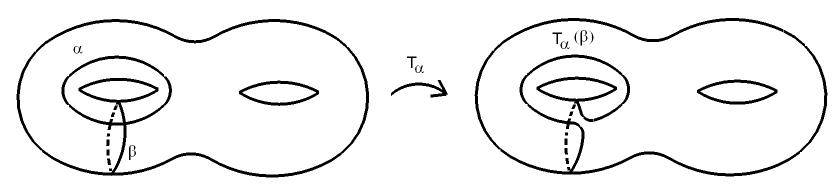}
	\caption{{\small The effect of the Dehn twist $T_{\alpha}$ on the simple closed curve $\beta$.}}
	\label{Dehn twist}
\end{figure}
 
Let $a$ represent the isotopy class of $\alpha$. The isotopy class (or mapping class) of the homeomorphism $T_{\alpha}$ is an element of Mod(S). We shall denote this element by $T_a$ and refer to it as the (left) Dehn twist along $a$. In what follows, a Dehn twist will always mean an element of Mod(S). 
 \medskip
 
If $a$ and $b$ are isotopy classes of simple closed curves in $S$, then the geometric intersection number $i(a,b)$ is the minimal number of intersection points between the representatives of $a$ and $b$. That is, 
\[
   i(a,b) = min_{\stackrel{\alpha \in a}{\beta \in b}}|\alpha \cap \beta|
\]

Let $a$ and $b$ represent isotopy classes of simple closed curves in $S$, and denote by $T_a$ and $T_b$ their respective (left) Dehn twists in Mod(S).

\begin{fact} \label{f1}
$T_a = T_b \Leftrightarrow a = b$.
\end{fact}
 \medskip

\begin{fact} \label{f2}
$T_a$ has infinite order.
\end{fact}
 \medskip
 
\begin{fact} \label{f4}
If $f \in Mod(S)$, then $fT_af^{-1} = T_{f(a)}$.
\end{fact}
 \medskip
 
The following fact follows easily from Facts ~\ref{f1} and ~\ref{f4}.
 
\begin{fact}\label{f}
Let $f \in Mod(S)$. Then $f T_a = T_a f \Leftrightarrow f(a) = a$.
\end{fact}
 
\begin{fact} \label{f3}
If $n$ is an integer, then $i(T_a^n(b),b) = |n|i(a,b)^2$.
\end{fact}
 \medskip

\begin{fact}\label{f5}
$T_aT_b = T_bT_a \Leftrightarrow i(a,b) = 0$. The left hand side of the equivalence is called the commutativity of disjointness relation. 
\end{fact}
 \medskip

\begin{fact}\label{f6}
If $a$ and $b$ are distinct, then $T_aT_bT_a = T_bT_aT_b \Leftrightarrow i(a,b) = 1$. The left hand side of the equivalence is known as the braid relation.
\end{fact}

The following fact is an easy consequence of Facts~\ref{f1}, ~\ref{f4}, and ~\ref{f6}. 

\begin{fact}\label{f7}
If $i(a,b) = 1$, then $T_aT_b(a) = b$ and $T_bT_a(b) = a$. 
\end{fact}
 \medskip
 
\begin{fact} \label{f8}
If $a \neq b$ and $T_a^p = T_b^q$ for some $p,q \in \mathbb{Z}$, then $p = q = 0$.
\end{fact}
 \medskip

\begin{thm}[2-Chain Relation]\label{chain relation}
If $i(a,b) = 1$, then $(T_aT_b)^6 = T_c$ in Mod(S), where $c$ is the boundary of a closed regular neighborhood of $a \cup b$. In particular, $T_c$ is trivial when $S = S_{1,0}$ and so $(T_aT_b)^6 = 1$ in this case. 
\end{thm}
 \medskip

\begin{thm}\label{two Dehn twist subgroups}
Denote by $\Gamma$ the subgroup of Mod(S) generated by $T_a$ and $T_b$. Then 
\begin{list}{\labelitemi}{\leftmargin=1em}
\item $\Gamma \cong \mathbb{Z}^2$ if $i(a,b) = 0$
\item $\Gamma \cong SL_2(\mathbb{Z})$ if $i(a,b) = 1$ and $S = S_{1,0}$
\item $\Gamma \cong \mathcal{B}_3$ if $i(a,b) = 1$ and $S \neq S_{1,0}$
\item $\Gamma \cong \mathbb{F}_2$ if $i(a,b) \geq 2$ (Ishida [\ref{I}])
\end{list}
\end{thm}

\section{Group Theory}
 \label{Group Theory}
 
Given a free group on finitely many generators and a finite index subgroup $H$, the following theorem (Theorem 2.10 in [\ref{MKS}]), implies that $H$ is a free group and determines the number of its generators.

\begin{thm}\label{free group theorem}
Consider the free group $\mathbb{F}_p$ on $p$ generators and let $H$ be an index $q$ subgroup in $\mathbb{F}_p$. If $p$ and $q$ are finite, then $H$ is a free group on $q(p-1) + 1$ generators.
\end{thm}

\begin{thm}
Suppose that $G$ is a virtually abelian group and let $H$ be a subgroup of $G$. Then $H$ is virtually abelian
\end{thm}
\begin{proof}
$G$ has a finite index subgroup $K$ which is abelian. Since $[H : H \cap K] \leq [G : K ]$, it follows that $H \cap K$ has finite index in $H$. Moreover, $H \cap K$ is abelian as it is a subgroup of $K$. Therefore, $H$ is virtually abelian.
\end{proof}

\begin{cor}\label{B3 is not va}
The braid group $\mathcal{B}_3$ is not virtually abelian.
\end{cor}
\begin{proof}
Choose $a$ and $b$ in $S \neq S_{1,0}$ so that $i(a,b) = 1$. By Theorem~\ref{two Dehn twist subgroups}, the group generated by $T_a$ and $T_b$ is isomorphic to $\mathcal{B}_3$. By Fact~\ref{f3}, $i(T_a^2(b),b) = 2$, and it follows from Theorem~\ref{two Dehn twist subgroups} that $T_{T_a^2(b)}$ and $T_b$ generate the free group $\mathbb{F}_2$ of order two. Since $\mathbb{F}_2$ is not virtually abelian, $\langle T_a,T_b \rangle$, and consequently $\mathcal{B}_3$, is not virtually abelian. 
\end{proof}

\begin{cor}\label{modular group is not va}
The modular group $SL_2(Z)$ is not virtually abelian. 
\end{cor}
\begin{proof}
Choose $a$ and $b$ in $S = S_{1,0}$ so that $i(a,b) = 1$. $T_a$ and $T_b$ generate Mod(S), which is known isomorphic to $SL_2(Z)$ (see [\ref{FM}]). As in the proof of Corollary~\ref{B3 is not va}, $T_{T_a^2(b)}$ and $T_b$ generate $\mathbb{F}_2$, which not virtually abelian. Therefore, $SL_2(Z)$ is not virtually abelian.
\end{proof}

\section{Proof of the Main Theorem}
 \label{proof of the main theorem}
 
Recall that $X \in \{(T_aT_b)^k,(T_bT_a)^k\}$, $k \in \mathbb{Z}$, and $Y = T_a$.
 \medskip
 
If $k = 0$, then $X$ is trivial and $G = \langle X,Y \rangle = \langle Y \rangle$. By Fact~\ref{f2}, $Y$ has infinite order, and so $G = \langle T_a \rangle \cong \mathbb{Z}$.
 \medskip
 
Assume $k \neq 0$.
 \medskip
 
If $i(a,b) = 0$, then $\langle T_a,T_b \rangle \cong \mathbb{Z}^2$ by Theorem~\ref{two Dehn twist subgroups}. In particular, $X = T_a^k T_b^k$ and hence, $G = \langle T_a,T_b^k \rangle$. Since $T_a$ and $T_b$ commute, every element in $G$ can be expressed in the form $T_a^{\alpha} T_b^{\beta}$ for some $\alpha, \beta \in \mathbb{Z}$. If $T_a^{\alpha} T_b^{\beta} = 1$, it follows from Fact~\ref{f8} that $\alpha = \beta = 0$. As such, $G$ is torsion free. This fact combined with $[X,Y] = 1$ imply that $G \cong \mathbb{Z}^2$. Moreover, $\langle T_a,T_b \rangle / \langle T_a,T_b^k \rangle \cong \mathbb{Z}_k$ and so $G = \langle T_a,T_b^k \rangle$ has index $k$ in $\langle T_a,T_b \rangle$.
 \medskip
 
Now suppose that $i(a,b) \geq 2$. By Theorem~\ref{two Dehn twist subgroups}, $\langle T_a,T_b \rangle \cong \mathbb{F}_2$. Since $G$ is a two generator subgroup of the free group $\mathbb{F}_2$, $G$ is either infinite cyclic or free on two generators. As the generators $X$ and $Y$ of $G$ have no common powers, $G$ is isomorphic to $\mathbb{F}_2$. If $k = \pm 1$, then $G = \langle T_a,(T_aT_b)^{\pm 1} \rangle = \langle T_a,T_b \rangle$. Now assume $k \neq \pm 1$. If $G$ has finite index $q$ in $\langle T_a,T_b \rangle$, then $q = 1$ by Theorem~\ref{free group theorem}. But $\langle T_a,T_b \rangle = \langle T_a, T_aT_b \rangle$ and $\langle T_a,T_aT_b \rangle$ modulo the normal closure of $\langle T_a,(T_aT_b)^k \rangle$ is isomorphic to the cyclic group $\mathbb{Z}_k$ of order $k$. As such, $G$ is a proper subgroup in $\langle T_a,T_b \rangle$ and so $q \neq 1$; a contradiction. Therefore, $G$ has infinite index in $\langle T_a,T_b \rangle$. The case when $X = (T_bT_a)^k$ and $Y = T_a$ follows similarly.
 \medskip
 
As shown above, note that $G$ and  $\langle T_a,T_b \rangle$ are isomorphic groups when $i(a,b) \neq 1$ and $k \neq 0$. More precisely, $G$ and $\langle T_a,T_b \rangle$ are both isomorphic to $\mathbb{Z}^2$ when $i(a,b) = 0$ and $k \neq 0$, and both groups are isomorphic to $\mathbb{F}_2$ when $i(a,b) \geq 2$ and $k \neq 0$.
 \medskip
 
The proof of the Main Theorem when $i(a,b) = 1$ is done through Lemma~\ref{lemma1}, Proposition~\ref{lemma2}, and Lemma~\ref{lemma3}. Lemma~\ref{lemma1} shows that conjugating $Y$ with $X$ depends primarily on $k$ and the conjugation can be easily determined once we specify the residue of $k$ modulo three. Proposition~\ref{lemma2} shows that $G$ is equal to $\langle T_a,T_b \rangle$ whenever $i(a,b) = 1$ and $k$ is not a multiple of three. Finally, Lemma~\ref{lemma3} investigates the structure of $G$ when $k$ is a multiple of three.

\begin{lem}\label{lemma1}
Let $k$ be a positive integer and suppose that $a$ and $b$ are isotopy classes such that $i(a,b) = 1$. Then 

\begin{table}[h!]
\begin{center}
 \begin{tabular}{|c|c|c|c|}
  \hline
 
 & 
$k \equiv 0 \hspace{0.1cm} mod(3)$ 
 & 
$k \equiv 1 \hspace{0.1cm} mod(3)$ 
 & 
$k \equiv 2 \hspace{0.1cm} mod(3)$ \\
  \hline 
$(T_aT_b)^k T_a (T_aT_b)^{-k}$ 
 &
$T_a$
 & 
$T_b$
 & 
$T_a T_b T_a^{-1}$ \\
  \hline
$(T_aT_b)^{-k} T_a (T_aT_b)^k$
 &
$T_a$
 &
$T_a T_b T_a^{-1}$
 &
$T_b$ \\
  \hline
$(T_bT_a)^k T_a (T_bT_a)^{-k}$ 
 &
$T_a$
 & 
$T_b T_a T_b^{-1}$
 & 
$T_b$ \\
  \hline
$(T_bT_a)^{-k} T_a (T_bT_a)^k$ 
 &
$T_a$
 & 
$T_b$
 & 
$T_a^{-1} T_b T_a$ \\
  \hline
 \end{tabular}
 \label{table1}
\end{center}
\end{table}     
\end{lem}
\begin{proof}
We only prove the first row of Table~\ref{table1}. The remaining rows can be shown similarly. We proceed by induction on $k$. 
\begin{eqnarray*}
\hspace{-1.7cm} k = 1: (T_a T_b) T_a (T_a T_b)^{-1} &=& T_a T_b T_a T_b^{-1} T_a^{-1}\\
                                                    &=& T_b T_a T_b T_b^{-1} T_a^{-1}\\  
                                                    &=& T_b
\end{eqnarray*}
\begin{eqnarray*}
\hspace{-1.5cm} k = 2: (T_a T_b)^2 T_a (T_a T_b)^{-2} &=& (T_a T_b) T_b (T_a T_b)^{-1}\\
                                                      &=& T_a T_b T_a^{-1}
\end{eqnarray*}
\begin{eqnarray*}
k = 3: (T_a T_b)^3 T_a (T_a T_b)^{-3} &=& (T_a T_b)(T_a T_b T_a^{-1})(T_a T_b)^{-1}\\
                                      &=& T_a T_b T_b^{-1} T_a T_b T_b^{-1} T_a^{-1}\\
                                      &=& T_a
\end{eqnarray*}
\noindent where the third equality is a consequence of the braid relation.

This takes care of the base case. Assume that the first row holds for some $k \geq 4$. Then
\begin{eqnarray*}
(T_a T_b)^{k+1} T_a (T_a T_b)^{-(k+1)} &=& (T_a T_b)(T_a T_b)^k T_a (T_a T_b)^{-k}(T_a T_b)^{-1}\\  
                                       &=& \left\{
   \begin{array}{lll}
     T_b & \mbox{if $k \equiv 0 \hspace{0.1cm} mod(3)$} \\
     T_a T_b T_a^{-1} & \mbox{if $k \equiv 1 \hspace{0.1cm} mod(3)$}\\
     T_a & \mbox{if $k \equiv 2 \hspace{0.1cm} mod(3)$}
   \end{array}
  \right. \\
\end{eqnarray*}
\end{proof}

\begin{prop}\label{lemma2}
Suppose that $a$ and $b$ are isotopy classes of simple closed curves in $S$ such that $i(a,b) = 1$. If $k \not \equiv 0 \hspace{0.1cm} mod(3)$, then $G = \langle T_a,T_b \rangle$. By Theorem~\ref{two Dehn twist subgroups}, this implies that $G \cong SL_2(\mathbb{Z})$ when $S = S_{1,0}$ and $G \cong \mathcal{B}_3$ when $S \neq S_{1,0}$.
\end{prop}
\begin{proof}
Recall that $X \in \{(T_aT_b)^k,(T_bT_a)^k\}$, $k \in \mathbb{Z}$, and $Y = T_a$. Assume that $i(a,b) = 1$ and $k \not \equiv 0 \hspace{0.1cm} mod(3)$. The following table shows how to generate all of $\langle T_a,T_b \rangle$ from $X$ and $Y$. More precisely, the table indicates how to obtain $T_b$ from $X$ and $Y = T_a$. This implies that $G = \langle X,Y \rangle = \langle T_a,T_b \rangle$. For example, if $X = (T_aT_b)^k$, $Y = T_a$, $k > 0$, and $k \equiv 1 \hspace{0.1cm} mod(3)$, then $XYX^{-1} = T_b$ according to the table below. That $XYX^{-1} = T_b$ follows from Lemma~\ref{lemma1}. If, on the other hand, $X = (T_aT_b)^k$, $Y = T_a$, $k > 0$, and $k \equiv 2 \hspace{0.1cm} mod(3)$, then $Y^{-1}XYX^{-1}Y = T_b$ according to the table. To see why $Y^{-1}XYX^{-1}Y = T_b$, note that $XYX^{-1} = T_a T_b T_a^{-1}$ by Lemma~\ref{lemma1} and recall that $Y = T_a$. The remaining entries in the table below can be checked in a similar fashion.

\begin{table}[h!]
\begin{center}
 \begin{tabular}{|c|c|c|c|}
  \hline
$X$
 & 
$Y$ 
 & 
$k \equiv 1 \hspace{0.1cm} mod(3)$ 
 & 
$k \equiv 2 \hspace{0.1cm} mod(3)$ \\
  \hline 
$(T_aT_b)^k$ 
 &
$T_a$
 & 
$XYX^{-1}$
 & 
$Y^{-1}XYX^{-1}Y$ \\
  \hline
$(T_aT_b)^{-k}$
 &
$T_a$
 &
$Y^{-1}XYX^{-1}Y$
 &
$XYX^{-1}$ \\
  \hline
$(T_bT_a)^k$ 
 &
$T_a$
 & 
$YXYX^{-1}Y^{-1}$
 & 
$XYX^{-1}$ \\
  \hline
$(T_bT_a)^{-k}$ 
 &
$T_a$
 & 
$XYX^{-1}$
 & 
$YXYX^{-1}Y^{-1}$ \\
  \hline
 \end{tabular}
  \label{table2}
\end{center}
\end{table}  
   
\end{proof}

\begin{lem}\label{lemma3}
Consider $a$ and $b$ such that $i(a,b) =1$.
\begin{list}{\labelitemi}{\leftmargin=1em}
\item If $S \neq S_{1,0}$ and $k \equiv 0 \hspace{0.1cm} mod(3)$, then $G \cong \mathbb{Z}^2$.
\item If $S = S_{1,0}$ and $k \equiv 0 \hspace{0.1cm} mod(6)$, then $G = \langle T_a \rangle \cong \mathbb{Z}$.
\item If $S = S_{1,0}$ and $k \equiv 3 \hspace{0.1cm} mod(6)$, then $G \cong \mathbb{Z}_2 \times \mathbb{Z}$.
\end{list}
\end{lem}
\begin{proof}
First assume that $i(a,b) = 1$ in $S \neq S_{1,0}$. By Theorem~\ref{two Dehn twist subgroups}, $\langle T_a,T_b \rangle \cong \mathcal{B}_3$. It is well known [\ref{KT}] that the center of $\langle T_a,T_b \rangle$ is infinite cyclic, generated by $(T_aT_b)^3$. Moreover, it is an immediate consequence of the braid relation that 
\[
   (T_aT_b)^3 = (T_bT_a)^3 \hspace{1.5cm} (\ast)
\]
\noindent So $(T_bT_a)^3$ generates the center of $\langle T_a,T_b \rangle$ as well. As such, $[X,Y] = 1$ for all $X \in \{(T_aT_b)^k,(T_bT_a)^k\}$, $k = 3n$ with $n \neq 0$, and $Y = T_a$. So every element in $G$ can be expressed in the form $X^{\alpha} Y^{\beta}$ for some $\alpha,\beta \in \mathbb{Z}$. For all nonzero $\beta$, note that $X^{\alpha} Y^{\beta}(b) \neq b$. To see this, observe that $X^{\alpha}$ is central in $\langle T_a,T_b \rangle$ and thus commutes with both $Y^{\beta}$ and $T_b$. As such, $X^{\alpha} Y^{\beta}(b) = Y^{\beta} X^{\alpha}(b) = Y^{\beta}(b)$ where the last equality is due to Fact~\ref{f}. Moreover, since $i(Y^{\beta}(b),b) = |\beta|i(a,b)^2 = |\beta| > 0$, $Y^({\beta}) \neq b$. $X^{\alpha} Y^{\beta}(b) \neq b$ combined with the fact that $\langle X \rangle$ is infinite cyclic imply that $G$ is torsion free. Consequently, $G \cong \mathbb{Z}^2$. Further, since $\langle T_a,T_b \rangle \cong \mathcal{B}_3$ is not virtually abelian (by Corollary~\ref{B3 is not va}), the abelian subgroup $G$ must have infinite index. 
 \medskip
 
Now assume that $i(a,b) = 1$ in the torus $S = S_{1,0}$. By Theorem~\ref{two Dehn twist subgroups}, $\langle T_a,T_b \rangle \cong SL_2(\mathbb{Z})$. It still follows from the braid relation that $(\ast)$ holds. It is easy to check that $(T_aT_b)^3$ is a nontrivial mapping class. Moreover, it follows from Theorem~\ref{chain relation} that $(T_aT_b)^6 = 1$. As such, $(T_aT_b)^3$ has order two. Thus, $X$ equals the identity when $k \equiv 0 \hspace{0.1cm} mod(6)$ and is an involution when $k \equiv 3 \hspace{0.1cm} mod(6)$. In the first case, it is immediate that $G = \langle Y \rangle \cong \mathbb{Z}$. In the second case, $X$ is an order two element which commutes with the infinite order element $Y$. Therefore, $G \cong \mathbb{Z}_2 \times \mathbb{Z}$. Finally, Corollary~\ref{modular group is not va} implies that $\langle T_a,T_b \rangle \cong SL_2(\mathbb{Z})$ is not virtually abelian. As such, the abelian subgroups $G$ that are isomorphic to $\mathbb{Z}$ or $\mathbb{Z}_2 \times \mathbb{Z}$ cannot have finite index. 
\end{proof}

\medskip

{\scriptsize
\noindent DEPARTMENT OF MATHEMATICS, COLLEGE OF COASTAL GEORGIA, BRUNSWICK, GA 31520, USA
 \smallskip
 
\noindent E-mail address: \textbf{jmortada@ccga.edu}}


\begin{thebibliography}{99}

\bibitem{Birman}\label{B}
Birman, J.
\emph{Braids, links and mapping class groups},
Annals of Mathematics Studies, 82, Princeton University Press, 1974.

\bibitem{Bosma-Cannon-Playoust}\label{BCP}
Bosma, W., Cannon, J., and Playoust, C.
\emph{The Magma algebra system. I. The user language},
J. Symbolic Comput., 24 (1997), 235–265. 

\bibitem{Wilhelm Magnus, Abraham Karrass, Donald Solitar}\label{MKS}
Magnus, W., Karrass, A., Solitar, D.
\emph{Combinatorial group theory: presentations of groups in terms of generators and relations},
Dover Publications, 2004.

\bibitem{Farb and Margalit}\label{FM}
Farb and Margalit.
\emph{A Primer on Mapping Class Groups}. Available at
http://www.math.utah.edu/~margalit/primer/

\bibitem{Ishida}\label{I}
Ishida, A.
\emph{The structure of subgroup of mapping class group generated by two Dehn twists},
Proc. Japan Acad. 72 (1996), 240-241.

\bibitem{Ivanov}\label{Iv}
Ivanov, N.
\emph{Mapping class groups},
Handbook of geometric topology, Noth-Holland, Amsterdam, 2002, 523-633.

\bibitem{Kassel-Turaev}\label{KT}
Kassel, C. and Turaev, V.
\emph{Braid Groups},
Graduate Texts in Mathematics, vol. 247, Springer, New York, 2008.

\bibitem{Margalit}\label{M}
Margalit, D.
\emph{A Lantern Lemma},
preprint, arXiv:math/0206120v2 [math.GT]
\end{thebibliography}
\end{document}